\documentclass[11pt,oneside]{amsart}

\usepackage{url}
\usepackage[english]{babel}
\usepackage{amsthm, amsmath, amssymb, amsfonts, amscd}
\usepackage[svgnames]{xcolor} %nonstandard colors
\usepackage[colorlinks,citecolor=DarkViolet,urlcolor=DarkViolet, bookmarksdepth=3, linkcolor=DarkCyan]{hyperref} %hyperrefs colors
\usepackage{kpfonts} %fonts

\usepackage{stmaryrd}
\usepackage{listings}
\usepackage{graphicx}
\usepackage{tikz-cd}
\usepackage{indentfirst}
\usepackage{mathrsfs}
\usepackage{accents}
\usepackage{array}

\usepackage{xcolor}

\def\ext{\operatorname{ext}}

\def\k{k}
\def\tk{\tilde \k}
\def\m{\mathfrak m}

\def\vol{\operatorname{vol}}
\def\coker{\operatorname{coker}}

\newtheorem{theorem}{Theorem}[section]

\newtheorem{lemma}[theorem]{Lemma}

\title{Parseval--Rayleigh identities for graded Artinian Gorenstein algebras}

\author{Mykola Pochekai}


\address{Mykola Pochekai \\ Center for Geometry and Physics \\ Institute for Basic Science \\ Pohang \\ South Korea}

\email{mykola.pochekai@gmail.com}

\begin{document}
\begin{abstract}
We formulate and prove Parseval--Rayleigh identities for graded Artinian Gorenstein algebras over fields of positive characteristic. Specializing the general result, we provide an alternative proof of the Parseval--Rayleigh identities of generic Artinian reductions of Stanley--Reisner rings of oriented simplicial spheres.
\end{abstract}

\maketitle
\markboth{}{}
\thispagestyle{empty}

\section{Introduction}
\begin{sloppypar}
Parseval--Rayleigh identities were introduced in the works \cite[Section 5]{AdiBeyondPositivity}, \cite[Section 7, Appendix A]{AdiPoly}, \cite{AdiCI} and have been used to establish Lefschetz properties for several classes of Artinian Gorenstein algebras, including generic complete intersections \cite[Corollary 4.7]{AdiCI} and generic Artinian reductions of IDP reflexive polytopes \cite[Theorem 1.2]{AdiPoly}. This suggests that similar identities should exist for any graded Artinian Gorenstein $\k$-algebra. We formulate and prove such identities.
\end{sloppypar}

In order to state a theorem, we assume that $\k$ is a field of characteristic $p$, $S = \k[x_0,\ldots,x_{m-1}]$ is a polynomial ring, and $I = (g_0,\ldots,g_{n-1})$ is a homogeneous Artinian Gorenstein ideal, where $g_0,\ldots,g_{n-1}$ is a fixed minimal homogeneous generating set of $I$, which means that $R = S/I$ is a Gorenstein Artinian $\k$-algebra of socle degree $s$. Let $\circ : S \times S \to \k$ be the Kronecker pairing, that is, 
$$
x^\alpha\circ x^\beta=
\begin{cases}
1,& \alpha=\beta,\\
0,& \alpha\neq \beta,
\end{cases}
$$
and then extend $\circ$ to polynomials by bilinearity. Let $M_c$ be the set of monic monomials of degree $c$. We assume that an isomorphism $\vol : R_s \to \k$ is provided. Using $\vol$ we can define $\nu$ as the unique nonzero class $\nu \in R_s$ so that $\vol(\nu) = 1$. Let $I^{[p]} = (g_0^p,\ldots,g_{n-1}^p) \subset S$, the algebra $\hat R =S/I^{[p]}$ is Gorenstein Artinian of socle degree $\hat s = ps + m(p-1)$. We can define $\hat \vol : \hat R_{\hat s} \to \k$ to be the unique functional which satisfies 
$$\hat \vol(x_0^{p-1} \ldots x_{m-1}^{p-1} \nu^p) = 1.$$
A quick remark: we will try to avoid writing natural projections and other type-casting notation in order not to overload the exposition. For example, the expression $\nu^p$ in the formula above actually means $\pi((\nu')^p)$, where $\pi : S \to S/I^{[p]}$ is the natural projection and $\nu'$ is any lift of $\nu$ from $R_s$ to $S_s$. This is well-defined: the class $\pi((\nu')^p)$ is independent of the choice of the lift $\nu'$. Statements of this kind, together with their proofs, will be made explicit in the text. With all that in mind, we can formulate our main theorem:

\begin{theorem} \label{th:main}
There is a unique element $\varepsilon\in ((I^{[p]}:I)/I^{[p]})_{\hat s-s}$
satisfying $\hat\vol(\varepsilon\nu)=1$. The next identity holds for every polynomial $w \in S_s$ of degree $s$:
    $$
    \vol(w)=\sum_{u\in M_s}
    ((x_0^{p-1}\cdots x_{m-1}^{p-1}u^p)\circ(\varepsilon w))
    \vol(u)^p.
    $$
\end{theorem}

The theorem is proved in Lemma \ref{lm:eps-in-ip-i}, Lemma \ref{lm:hat-vol-nu-eps}, Theorem \ref{th:submain}. Moreover, the proof shows that $\varepsilon$ is the polynomial representing the top component
$$P^{-m}:\hat F^{-m}\to F^{-m}$$
of a chain map between minimal free resolutions lifting the canonical projection $\hat R\to R$; see equation \eqref{eq:varepsilon-def} (in fact, this is how we define $\varepsilon$ in the main text).

In Section \ref{sec:examples}, we explain how the general theorem specializes to the previously known cases and then use it to produce a genuinely new example. First, in the complete intersection case (Subsection \ref{subsec:CI}), we recover the formula of \cite[Theorem 1.1]{AdiCI}. Second, for generic Artinian reductions of Stanley--Reisner rings of oriented simplicial spheres (Subsection \ref{subsec:spheres}), we recover the simplicial-sphere case of the Parseval--Rayleigh identities of \cite[Theorem A.1, Appendix A]{AdiPoly}. In this case our argument gives a new proof of these identities: we obtain them from the general homological construction of the element $\varepsilon$, which differs from the strategy in \cite{AdiPoly}. Finally, we compute $\varepsilon$ explicitly in homological terms for a specific non-complete-intersection graded Artinian Gorenstein $\k$-algebra (Subsection \ref{subsec:minimalnonci}), in order to demonstrate how the homological part of the proof works. The Parseval--Rayleigh identities for this example do not seem to be covered by the existing literature.

\subsection{Acknowledgments}
The author is grateful to Karim Adiprasito for giving him the opportunity to think about this problem during a visit to Sorbonne in November 2025. He also thanks Ryoshun Oba, Vasiliki Petrotou, and Mingzhi Zhang for extensive discussions during that visit.

\section{Proof of the main theorem} \label{sec:main}
All modules are $\mathbb Z$-graded. If $S=\k[x_0,\ldots,x_{m-1}]$ is a polynomial ring and $M$ is an $S$-module, we write $M_k$ for the $k$-th graded piece of $M$, and we use the notation $M(a)$ to denote the module with grading shifted by $a$, so $M(a)_d = M_{a+d}$. With this convention, a free $S$-module generated by an element of degree $a$ is abstractly isomorphic to $S(-a)$. By $S(-a)$ we mean a free graded $S$-module generated in degree $a$ with distinguished degree $a$ generator $\eta_a \in S(-a)$. We will use cohomological grading when we work with chain complexes, in order to ensure that it is not mixed up with internal grading. In other words, a complex of $S$-modules could be denoted by $F^{\bullet}$, and $F^{-5}_3$ would mean the $3$-rd graded piece of the $(-5)$-th term of the complex $F^{\bullet}$.

Let $\k$ be a field of positive characteristic $p$. Let $S = \k[x_0,\ldots,x_{m-1}]$ and let $I = (g_0,\ldots,g_{n-1}) \subset (x_0,\ldots,x_{m-1})$ be an Artinian Gorenstein homogeneous ideal, where $g_0,\ldots,g_{n-1}$ is a fixed minimal homogeneous generating set of $I$ (so, $n \geq m$); by this we mean that $S/I$ is a graded Artinian Gorenstein $\k$-algebra. Let $R = S/I$. Most of the time we omit writing the natural projection $\pi : S \to R$ explicitly, which means that if $f \in S_s$ is a polynomial and $\phi : R_s \to \k$ is a functional, we may write something like $\phi(f)$ without decorating the symbol $f$ with an overline or precomposing with $\pi$. By the Auslander-Buchsbaum formula, we have $m = pd_S(R)$, which implies that a minimal free $S$-resolution of $R$ has $m+1$ nonzero terms. Assume that $R$ has socle degree $s$, that is, $R_s \neq 0$ and $R_{i} = 0$ for every $i>s$. Since $R$ is Artinian Gorenstein, $R_s$ is a one-dimensional $\k$-vector space. Also assume that we are supplied with the choice of a $\k$-linear isomorphism $\vol : R_s \to \k$. Let $0 \to F^{-m} \to \ldots \to F^{-1} \to F^0 = S \to 0$ be a minimal free resolution of the $S$-module $R$. The map $\vol$ gives us a way to define a distinguished generator $c_0 \in F^{-m}$:
\begin{lemma} \label{lm:canonical-iso}
    There is a canonical isomorphism of the following one-dimensional $\k$-vector spaces:
    \begin{itemize}
        \item $\hom_{\k}(R(s), \k)_0$
        \item $\ext^m_S(R,S(-m-s))_0$
        \item $\hom_S(F^{-m},S(-m-s))_0$
    \end{itemize}
    The isomorphism of the first space with the last space gives us a canonical choice of an $S$-generator $c_0 \in F^{-m}$ of degree $m+s$, as long as the choice of $\vol$ is provided.
\end{lemma}

\begin{proof}
   We denote by $\m = (x_0,\ldots,x_{m-1}) \subset S$ the ideal generated by the variables. Graded local duality \cite[Section 21.11]{Eisenbud} gives a canonical isomorphism of graded $S$-modules
   $$\ext^m_S(R,S(-m)) \cong \hom_k(R,k),$$
   which, after twisting by $(-s)$ and taking $0$-th graded pieces, gives an isomorphism
   $$\ext^m_S(R,S(-m-s))_0 \cong \hom_k(R(s),k)_0.$$
   Here, the oriented basis $x_0,\ldots,x_{m-1}$ has been used to identify $\omega_S$ with $S(-m)$. Since $F^{\bullet}$ is a minimal graded free resolution of the $S$-module $R$, we have
   $$\ext^m_S(R,S(-m-s)) = \coker(\hom_S(F^{-(m-1)},S(-m-s)) \to \hom_S(F^{-m},S(-m-s))).$$
   There is the natural quotient map $\hom_S(F^{-m},S(-m-s)) \to \ext^m_S(R,S(-m-s))$. Since $F^{\bullet}$ is a minimal free resolution, we know that the differential $d : F^{-m} \to F^{-(m-1)}$ factors through $\m F^{-(m-1)}$. Since $\m$ has no elements of degree $0$, the $0$-th degree part is unaffected by the quotient, so the quotient map induces an isomorphism on $0$-th pieces
   $$\hom_S(F^{-m},S(-m-s))_0 \to \ext^m_S(R,S(-m-s))_0.$$
   Thus, there is a degree $0$ map $\phi_{\vol} \in \hom_S(F^{-m},S(-m-s))_0$ which corresponds to $\vol$, and one can choose $c_0 \in F^{-m}_{m+s}$ to be the unique element such that $\phi_{\vol}(c_0) = \eta_{m+s}$, where $\eta_{m+s}$ is the fixed generator of $S(-m-s)$. 
\end{proof}

Now, let $I^{[p]} = (g_0^p,\ldots,g_{n-1}^p)$ be the ideal in $S$. Let $\hat F^{\bullet} := F^{\bullet} \otimes_{S,Frob} S =  0 \to \hat F^{-m} \to \ldots \to \hat F^0 \to 0$. By Kunz's theorem, the Frobenius map $Frob : S \to S$ is flat, so $\hat F^{\bullet}$ is a minimal free resolution of an Artinian Gorenstein algebra $S/I^{[p]}$ of socle degree $\hat s = ps + m(p-1)$. We denote $\hat c_0 = c_0 \otimes 1 \in \hat F^{-m}$.

The quotient map $R \to \k$ gives rise to the chain map $N^{\bullet} : F^{\bullet} \to K^{\bullet}$ between the minimal free resolution $F^{\bullet}$ of the $S$-module $R$ and the Koszul complex $K^{\bullet}$, which is the minimal free resolution of the $S$-module $\k$. We assume that the linear order of the minimal generators $g_0,\ldots,g_{n-1}$ and variables $x_0,\ldots,x_{m-1}$ is fixed; this gives an ordered basis for $F^{-1}$ and $K^{-1}$. We say that $a_0,\ldots,a_{n-1}$ is a basis of $F^{-1}$ and $e_0,\ldots, e_{m-1}$ is a basis of $K^{-1}$. Since the order of variables is given, we can identify $K^{-m} = S(-m)$ with the distinguished generator $e:=e_0 \wedge \ldots \wedge e_{m-1}$. The map $N^{-m} : F^{-m} \to K^{-m}$ is now identified with a map $S(-m-s) \to S(-m)$ and hence with a polynomial $\nu$ of degree $s$. In other words, we have the identity $N^{-m}(c_0) = \nu e$. The polynomial $\nu$ depends on the choice of $N^{\bullet}$; however, its class in $S/I$ is canonical, as the next lemma states:

\begin{lemma} \label{lm:vol-nu}
    We have $\vol(\nu) = 1$. In particular, the class $[\nu] \in R_s$ is well-defined independently of the choice of lifting $N^{\bullet} : F^{\bullet} \to K^{\bullet}$.
\end{lemma}
\begin{proof}
    The map $N^{-m} : F^{-m} \to S(-m)$ represents a class in $\ext^m_S(R,S(-m))$. By graded local duality \cite[Section 21.11]{Eisenbud}, $\ext^m_S(R,S(-m))_0 = \hom_{\k}(R,\k)_0$, and under this identification the quotient class of the map $N^{-m} : F^{-m} \to K^{-m}$ in $\ext^m_S(R,S(-m))$ corresponds exactly to the quotient map $R \to \k$. In order to show that $\vol(\nu)=1$, recall how $c_0$ was chosen. 
    The functional
    $$\vol\in \hom_k(R(s),k)_0$$
    corresponds, by the Lemma \ref{lm:canonical-iso}, to a degree $0$ map
    $$\phi_{\vol}:F^{-m}\to S(-m-s)$$
    satisfying
    $$\phi_{\vol}(c_0)=\eta_{m+s},$$
    where $\eta_{m+s}$ is the distinguished generator of $S(-m-s)$. Since $N^{-m}(c_0)=\nu e$, the map $N^{-m}$ is the composition
    $$F^{-m}\xrightarrow{\phi_{\vol}} S(-m-s)\xrightarrow{\cdot \nu} S(-m),$$
    where the second map is multiplication by $\nu$. Under graded local duality, this composition corresponds to the functional
    $$R\to k,\qquad f \mapsto \vol(\nu f).$$
    Evaluating this functional at $1\in R_0$, we get
    $$1\mapsto \vol(\nu).$$

    On the other hand, as explained above, the same $\ext$ class is represented by the chain map $N^\bullet:F^\bullet\to K^\bullet$, which lifts the quotient map $R\to k$. Hence the corresponding functional sends $1$ to $1$. Therefore
    $$\vol(\nu)=1.$$

    Finally, since $R_s$ is one-dimensional and $\vol:R_s\to k$ is an isomorphism, there is a unique class in $R_s$ whose volume is equal to $1$. Hence the class $[\nu]\in R_s$ is independent of the choice of the lifting $N^\bullet$.
\end{proof}

By homological algebra, there exists a chain map $P^{\bullet} : \hat F^{\bullet} \to F^{\bullet}$ lifting the quotient map $S/I^{[p]} \to S/I$. The algebra $\hat R = S/I^{[p]}$ is a graded Gorenstein Artinian $\k$-algebra with socle degree $\hat s = ps + m (p-1)$. Since the generators $c_0$ and $\hat c_0$ have been chosen, the map $P^{-m} : \hat F^{-m} \to F^{-m}$ gives a polynomial $\varepsilon \in S_{\hat s - s}$ of degree $\hat s - s$ such that 
\begin{equation}\label{eq:varepsilon-def}
    P^{-m}(\hat c_0) = \varepsilon c_0.
\end{equation}
The polynomial $\varepsilon$ depends on the choice of the chain map $P^{\bullet}$; however, the class of $\varepsilon$ in the one-dimensional $\k$-vector space $((I^{[p]}:I)/I^{[p]})_{\hat s - s}$ is well-defined and independent of the choice of the chain map $P^{\bullet}$.

\begin{lemma} \label{lm:eps-in-ip-i}
The class $[\varepsilon] \in ((I^{[p]}:I)/I^{[p]})_{\hat s - s}$ is well-defined and independent of the choice of a chain map $P^{\bullet}$ lifting the quotient map $S/I^{[p]} \to S/I$.
\end{lemma}
\begin{proof}
    By \cite[Corollary 21.16(b)]{Eisenbud}, the chain complex $(F^*)^\bullet:=\hom_S(F^{-m-\bullet},S(-m-s))$ is abstractly isomorphic to $F^\bullet$. We may choose an isomorphism such that its component $F^0 \to (F^*)^0=\hom_S(F^{-m},S(-m-s))$
    sends $1\in F^0=S$ to the map $F^{-m}\to S(-m-s)$ sending $c_0$ to the distinguished generator $\eta_{m+s}$ of $S(-m-s)$. Let $a_0,\ldots,a_{n-1}$ be the chosen basis of $F^{-1}$, so that
    $d_F^{-1}(a_i)=g_i$.
    Under the self-duality isomorphism, $F^{-(m-1)}$ is identified with
    $$(F^*)^{-(m-1)}=\hom_S(F^{-1},S(-m-s)).$$
    Let $b_0,\ldots,b_{n-1}$ be the basis of $F^{-(m-1)}$ corresponding to the dual basis $a_0^\vee,\ldots,a_{n-1}^\vee$, where $a_i^\vee(a_j)=\delta_{ij} \eta_{m+s}$. Here $\eta_{m+s}$ is the distinguished generator of $S(-m-s)$. Under this identification, the top differential $d_F^{-m}:F^{-m}\to F^{-(m-1)}$ corresponds to the dual of the first differential
    $d_F^{-1}:F^{-1}\to F^0=S$.
    Therefore, if $d_F^{-m}(c_0)=\sum_i h_i b_i$,
    then the coefficients $h_i$ are determined by dualizing the formula $d_F^{-1}(a_i)=g_i$.
    Indeed, the dual map $\hom_S(F^0,S(-m-s))\to \hom_S(F^{-1},S(-m-s))$ sends the map $1\mapsto \eta_{m+s}$ to the map $F^{-1}\to S(-m-s)$ given by $a_i\mapsto g_i \eta_{m+s}$. In the dual basis this map is exactly $\sum_i g_i a_i^\vee$. Hence, after transporting back through the self-duality isomorphism, we get
    $$d_F^{-m}(c_0)=\sum_i g_i b_i.$$
    In particular, the coordinates of $d_F^{-m}(c_0)$ generate the ideal $(g_0,\ldots,g_{n-1})=I$. Knowing that, the identities
    $$P^{-(m-1)} d_{\hat F}^{-m} (\hat c_0) = d^{-(m-1)}_F P^{-m}(\hat c_0),$$
    $$\sum_i g_i^p P^{-(m-1)}(\hat b_i) = \sum_i g_i \varepsilon  b_i,$$
    give us $\varepsilon I \subset I^{[p]}$, which proves $\varepsilon \in (I^{[p]} : I)$.

    Now we prove that the class modulo $I^{[p]}$ is independent of the choice of $P^{\bullet}$. Let $P^{\bullet}$ and $(P')^{\bullet}$ be two chain maps $\hat F^{\bullet}\to F^{\bullet}$ lifting the same quotient map $S/I^{[p]}\to S/I$. Since they lift the same map on homology, they are chain-homotopic. Hence there exists a homotopy $H$ such that
    $$P^{\bullet}-(P')^{\bullet}=d_FH+Hd_{\hat F}.$$
    Looking at degree $-m$, the term $d_FH$ vanishes because $F^{-m-1}=0$. Therefore
    $$(P^{-m}-(P')^{-m})(\hat c_0)=H^{-(m-1)}d_{\hat F}^{-m}(\hat c_0).$$
    But from the formula above we know that
    $$d_{\hat F}^{-m}(\hat c_0)=\sum_i g_i^p\hat b_i,$$
    hence $d_{\hat F}^{-m}(\hat c_0)\in I^{[p]}\hat F^{-(m-1)}$. Since $H^{-(m-1)}$ is $S$-linear, it follows that
    $$(P^{-m}-(P')^{-m})(\hat c_0)\in I^{[p]}F^{-m}.$$
    Write
    $$P^{-m}(\hat c_0)=\varepsilon c_0,\qquad (P')^{-m}(\hat c_0)=\varepsilon'c_0.$$
    Then
    $$(\varepsilon-\varepsilon')c_0\in I^{[p]}F^{-m}.$$
    Since $F^{-m}$ is a free $S$-module with generator $c_0$, this implies $\varepsilon-\varepsilon'\in I^{[p]}$. Therefore the class of $\varepsilon$ in $((I^{[p]}:I)/I^{[p]})_{\hat s-s}$ is independent of the choice of the lifting $P^{\bullet}$.
\end{proof}

Before proving the final result, we will prove the next lemma, which will be useful later and, in some sense, gives us a way to compute $\varepsilon$ in a non-homological way. Let $\hat \vol : \hat R_{\hat s} \to k$ be the unique linear functional satisfying
$$\hat \vol(x_0^{p-1}\cdots x_{m-1}^{p-1}\nu^p)=1.$$
Then the next lemma holds:

\begin{lemma} \label{lm:hat-vol-nu-eps}
    We have $\hat \vol(\nu\varepsilon)=1$.
\end{lemma}
\begin{proof}
Let $K^\bullet$ be the Koszul complex resolving $k$, and let $\hat K^\bullet=K^\bullet\otimes_{S,\operatorname{Frob}}S$ be the Koszul complex on $x_0^p,\ldots,x_{m-1}^p$. We have the chain map
$$\hat N^\bullet:\hat F^\bullet\to \hat K^\bullet$$
obtained from $N^\bullet$ by applying Frobenius. In top degree,
$$\hat N^{-m}(\hat c_0)=\nu^p(\hat e_0\wedge\cdots\wedge \hat e_{m-1}).$$

Let $D^\bullet:\hat K^\bullet\to K^\bullet$ be the standard comparison map lifting the quotient map $S/(x_0^p,\ldots,x_{m-1}^p)\to k$. In degree $-m$, it is given by
$$D^{-m}(\hat e_0\wedge\cdots\wedge \hat e_{m-1})
=
x_0^{p-1}\cdots x_{m-1}^{p-1}(e_0\wedge\cdots\wedge e_{m-1}).$$
Therefore
$$D^{-m}\hat N^{-m}(\hat c_0)
=
x_0^{p-1}\cdots x_{m-1}^{p-1}\nu^p(e_0\wedge\cdots\wedge e_{m-1}).$$
On the other hand, the composition $N^\bullet P^\bullet:\hat F^\bullet\to K^\bullet$ satisfies
$$N^{-m}P^{-m}(\hat c_0)
=
N^{-m}(\varepsilon c_0)
=
\varepsilon\nu(e_0\wedge\cdots\wedge e_{m-1}).$$
The two chain maps
$$N^\bullet P^\bullet,\qquad D^\bullet\hat N^\bullet:\hat F^\bullet\to K^\bullet$$
lift the same quotient map $S/I^{[p]}\to k$.  We demonstrate this by the following commutative diagrams: the right commutative-up-to-homotopy square is a chain lift of the left commutative square:

\[
\begin{tikzcd}
{S/I^{[p]}} && {S/(x_0^p,\ldots,x_{m-1}^p)} \\
{S/I}       && {S/(x_0,\ldots,x_{m-1})}
\arrow[from=1-1, to=1-3]
\arrow[from=1-1, to=2-1]
\arrow[from=2-1, to=2-3]
\arrow[from=1-3, to=2-3]
\end{tikzcd}
\qquad
\begin{tikzcd}
{\hat F^\bullet} && {\hat K^\bullet} \\
{F^\bullet}      && {K^\bullet}
\arrow[from=1-1, to=1-3, "\hat N^\bullet"]
\arrow[from=1-1, to=2-1, "P^\bullet"]
\arrow[from=2-1, to=2-3, "N^\bullet"]
\arrow[from=1-3, to=2-3, "D^\bullet"]
\end{tikzcd}
\]
Hence $N^\bullet P^\bullet$ and $D^\bullet\hat N^\bullet$ are chain-homotopic. Hence, in degree $-m$, their difference has the form $H^{-(m-1)}d_{\hat F}^{-m}$ for some $S$-linear map $H^{-(m-1)}:\hat F^{-(m-1)}\to K^{-m}$, because $\hat F^{-m-1}=0$. Evaluating on $\hat c_0$ and using $d_{\hat F}^{-m}(\hat c_0)\in I^{[p]}\hat F^{-(m-1)}$, we see that the difference of the two top-degree coefficients lies in $I^{[p]}$. Since $d_{\hat F}^{-m}(\hat c_0)\in I^{[p]}\hat F^{-(m-1)}$, we get
$$\varepsilon\nu-x_0^{p-1}\cdots x_{m-1}^{p-1}\nu^p\in I^{[p]}.$$
Thus these two polynomials define the same class in $\hat R=S/I^{[p]}$. Applying $\hat\vol$ gives
$$\hat\vol(\varepsilon\nu)
=
\hat\vol(x_0^{p-1}\cdots x_{m-1}^{p-1}\nu^p)
=
1,$$
by the definition of $\hat\vol$. This proves the lemma.
\end{proof}

\begin{theorem} \label{th:submain}
Let $M_s$ be the set of monic monomials in $m$ variables of degree $s$.
Let $\varepsilon \in ((I^{[p]}:I)/I^{[p]})_{\hat s-s}$ be defined as above. Then for every
$w\in S_s$ we have
$$
\vol(w)
=
\sum_{u\in M_s}
\left(
(x_0^{p-1}\cdots x_{m-1}^{p-1}u^p)\circ(\varepsilon w)
\right)
\vol(u)^p.
$$
\end{theorem}

\begin{proof}
Since $R_s=(S/I)_s$ is one-dimensional and generated by the class of $\nu$,
every element of $S_s$ is a linear combination of $\nu$ and elements of $I_s$.
Therefore it is enough to prove the identity in the following two cases:
first, when $w=\nu$, and second, when $w\in I_s$. We first consider the case $w\in I_s$. Since $I$ is generated by
$g_0,\ldots,g_{n-1}$, it is enough to consider elements of the form
$w=hg_i$, where $\deg(hg_i)=s$. Then $\vol(w)=0$, because $w$ maps to zero
in $R_s$. On the other hand, since $\varepsilon\in (I^{[p]}:I)$, by Lemma \ref{lm:eps-in-ip-i}, we have $\varepsilon w\in I^{[p]}.$
Therefore the right-hand side is zero by \cite[Lemma 3.2]{AdiCI}. For the sake of completeness we repeat the proof here. So we need to show that for any $h'$ such that $\deg( h' g_i^p) = \hat s$ the identity
$$\sum_{u \in M_s} ((x_0^{p-1} \ldots x_{m-1}^{p-1} u^p) \circ (h'g_i^p))\vol(u)^p = 0$$
holds. To do this we decompose $h'$ as a sum of monomials, only monomials which are of the form $x_0^{p-1} \ldots x_{m-1}^{p-1} (h'')^{p}$, for some monomial $h''$,  will survive, so we have
$$\sum_{u \in M_s} ((x_0^{p-1} \ldots x_{m-1}^{p-1} u^p) \circ (x_0^{p-1} \ldots x_{m-1}^{p-1} (h'')^p g_i^p))\vol(u)^p = 0$$
$$\sum_{u \in M_s} (u^p) \circ ((h'')^p g_i^p)) \vol(u)^p = 0$$
$$\vol(h'' g_i)^p = 0$$
the last identity is true since $h'' g_i \in \ker \vol$. Hence the identity holds for all $w \in I_s$. It remains to prove the identity for $w=\nu$. By the proof of Lemma \ref{lm:hat-vol-nu-eps}, we have
$$\varepsilon\nu-x_0^{p-1}\cdots x_{m-1}^{p-1}\nu^p\in I^{[p]}.$$
Thus, again by \cite[Lemma 3.2]{AdiCI}, the right-hand side for $w=\nu$ is
equal to
$$\sum_{u\in M_s} (u^p\circ \nu^p)\vol(u)^p.$$
Therefore it remains to prove
$$\vol(\nu)=\sum_{u\in M_s} (u^p\circ \nu^p)\vol(u)^p.$$
We may now use the same proof as in \cite[Theorem 1.1]{AdiCI}. Since the Kronecker pairing is compatible with Frobenius, we have $u^p\circ \nu^p = (u\circ \nu)^p$. Hence
$$\sum_{u\in M_s} (u^p\circ \nu^p)\vol(u)^p=\sum_{u\in M_s} (u\circ \nu)^p\vol(u)^p.$$
Since the field has characteristic $p$, this is equal to
$$\left(\sum_{u\in M_s} (u\circ \nu)\vol(u)\right)^p.$$
By linearity of $\vol$, we get
$$\left(\sum_{u\in M_s} (u\circ \nu)\vol(u)\right)^p = \left(\vol\left(\sum_{u\in M_s}(u\circ \nu)u\right)\right)^p.$$
But
$$\sum_{u\in M_s}(u\circ \nu)u=\nu,$$
because this is just the expansion of $\nu$ in the monomial basis of $S_s$.
Therefore the right-hand side is $\vol(\nu)^p$. By Lemma \ref{lm:vol-nu}, we have $\vol(\nu)=1$. Hence
$$\vol(\nu)^p=1^p=1=\vol(\nu).$$
This proves the identity for $w=\nu$, and therefore for every $w\in S_s$.
\end{proof}

\section{Examples} \label{sec:examples}

We explain how Theorem \ref{th:main} specializes to the two main known cases: complete intersections, where it recovers the formula of \cite[Theorem 1.1]{AdiCI}, and generic Artinian reductions of Stanley--Reisner rings of oriented simplicial spheres, where it recovers the formula of \cite[Appendix A, Theorem A.1]{AdiPoly}. We then give an explicit non-complete-intersection example to illustrate how the same method computes the polynomial $\varepsilon$ beyond the previously known cases.

\subsection{Complete intersections} \label{subsec:CI}
Let $S = \k[x_0,\ldots,x_{n-1}]$. The case when $I=(g_0,\ldots,g_{n-1})$ is a homogeneous complete intersection ideal has been described in \cite{AdiCI}. For complete intersections, we have $n=m$, so we will use the number $n$ for both, also we have socle degree of $S/I$ is equal to $s = -m + \sum_i \deg g_i$. As in Section \ref{sec:main}, we denote by $F^\bullet,\hat F^\bullet,K^\bullet,\hat K^\bullet$, respectively, the minimal free resolution of $S/I$, the minimal free resolution of $S/I^{[p]}$, the Koszul complex resolving $S/(x_0,\ldots,x_{n-1})$,
and the Koszul complex resolving $S/(x_0^p,\ldots,x_{n-1}^p)$. We denote the corresponding projection morphism by $\pi : S \to S/I$. We assume that the complete intersection ideal $I$ is provided with the choice of an ordered minimal homogeneous generating set, that is $I = (g_0,\ldots,g_{n-1})$. The $i$-th basis vectors of the $S$-modules $F^{-1}, \hat F^{-1}, K^{-1}, \hat K^{-1}$ will be denoted by $a_i,\hat a_i, e_i, \hat e_i$. Instead of choosing $\vol$, they choose the morphism $N^{-1} : F^{-1} \to K^{-1}$ as initial data, which commutes with differentials. This means that if $N^{-1}(a_i) =  \sum_j N^{-1}_{ij} e_j $, then $d_F^{-1} (a_i) = d_K^{-1} N^{-1} (a_i)$ can be read as $g_i = \sum_j N^{-1}_{ij} x_j$, or, in other words, $N^{-1}$ is an $n \times n$ matrix with entries in the polynomial ring $S$ such that
$$
\begin{pmatrix}
g_0\\
\vdots\\
g_{n-1}
\end{pmatrix}
=
N^{-1}
\begin{pmatrix}
x_0\\
\vdots\\
x_{n-1}
\end{pmatrix}.
$$
Compare with \cite[Formula 3]{AdiCI}. Notice that $N^{-1}$ here refers to cohomological grading; it does not denote an inverse operator.

Since both complexes $F^\bullet$ and $K^\bullet$ are Koszul complexes, the morphism $N^{-1}$ determines the whole morphism $N^{\bullet} : F^{\bullet} \to K^{\bullet}$ as long as it preserves products. In fact, in the top degree we have $N^{-n}(a_0 \wedge \ldots \wedge a_{n-1}) = N^{-1}(a_0) \wedge \ldots \wedge N^{-1}(a_{n-1}) = z_0 e_0 \wedge \ldots \wedge e_{n-1}$, where $z_0 = \det(N^{-1})$ (in our terminology, it is called $\nu = \det(N^{-1})$). This means that the corresponding $\vol : R_s \to \k$, by Lemma \ref{lm:vol-nu}, should be chosen as $\vol(z_0) = 1$. This is exactly their normalization; see \cite[Formula 5]{AdiCI}. 

There is also a natural choice for the morphism $P : \hat F^{\bullet} \to F^{\bullet}$, since, again, both complexes are Koszul complexes. We can define $P^{-1}(\hat a_i) = g_i^{p-1} a_i$, since this is the most obvious morphism $\hat F^{-1} \to F^{-1}$ which commutes with the differential, after that there is a unique morphism $P: \hat F^\bullet \to F^\bullet$ of comutative $dg$-algebras, so we can deduce that
$$P^{-n}(\hat a_0 \wedge \ldots \wedge \hat a_{n-1}) = g_0^{p-1} \ldots g_{n-1}^{p-1} a_0 \wedge \ldots \wedge a_{n-1}.$$ In other words, we have $\varepsilon = g_0^{p-1} \ldots g_{n-1}^{p-1}$. After that, the formula from Theorem \ref{th:main} becomes
$$\vol (w) = \sum_{u \in M_s} ((x_0^{p-1} \ldots x_{n-1}^{p-1} u^p) \circ (g_0^{p-1} \ldots g_{n-1}^{p-1} w)) \vol(u)^p,$$
which is precisely the formula in \cite[Theorem 1.1]{AdiCI}.

\subsection{Simplicial spheres} \label{subsec:spheres}
Let $X$ be an oriented simplicial sphere, which means a finite simplicial complex whose geometric realization is homeomorphic to the $n$-sphere, and for every $n$-facet $F$ there is a chosen positive ordering (that is, a total ordering up to even permutation) on the set of vertices of every facet $F = [v_0,\ldots,v_n]$. All orderings should be compatible in the sense that if two $n$-facets share the same $(n-1)$-face, then the induced orientations on that $(n-1)$-face should be opposite. By $X_0$ we denote the finite set of vertices of the simplicial complex $X$. Let $\k$ be a field of positive characteristic $p$. We denote by $\k[X] = \k[X_0]/I$ the Stanley-Reisner ring, so $I$ is the nonface ideal. By $\tk = \k(\theta_{ij})$ we mean the rational function field with indeterminates $\theta_{ij}$ for $0\leq i \leq n$ and $j \in X_0$. Let $\theta_i = \sum_j \theta_{ij} x_j \in \tk[X_0]$ and $A(X) = \tk[X]/(\theta_0,\ldots,\theta_n)$. The algebra $A(X)$ is a graded Artinian Gorenstein algebra of socle degree $s = n+1$. For oriented simplicial spheres, there is a good choice of volume $\vol : A(X)_s \to \tk$, called the Kustin-Miller volume, see \cite{AdiIntrinsic}. This is the unique volume which satisfies $\vol(\det(\theta_F)x_F) = 1$ for any oriented facet $F \in X_n$, where by $\det(\theta_F)$ we mean the determinant of the matrix $\theta_F = (\theta_{iv})_{0 \leq i \leq n, v \in F}$ where the vertices are taken in the positive order. Let us denote $\theta = (\theta_0,\ldots,\theta_n) \subset \tk [X_0]$. 

We can check that $\varepsilon = x_{X_0}^{p-1} \theta_0^{p-1} \ldots \theta_n^{p-1}$ by using Lemma \ref{lm:eps-in-ip-i} and Lemma \ref{lm:hat-vol-nu-eps}. Namely, it is easy to check that $\varepsilon \in (I^{[p]} + \theta^{[p]}: I + \theta)_{\hat s - s}$ since multiplying $x_{X_0}^{p-1}$ by a nonface monomial $x_K \in I$ gives $x_K^{p} x_{X_0 \setminus K}^{p-1} \in I^{[p]}$, and multiplying $\theta_0^{p-1} \ldots \theta_n^{p-1}$ by $\theta_i$ gives $\theta_0^{p-1} \ldots \theta_i^p \ldots \theta_n^{p-1} \in \theta^p$. By Lemma \ref{lm:hat-vol-nu-eps}, we need to prove that $\hat \vol(\varepsilon \nu) = 1$, where $\nu = \det(\theta_F) x_F$ for some positively oriented facet $F$, and $\hat \vol : (\tk [X_0]/(I^{[p]} + \theta^{[p]}))_{\hat s } \to \tk$ is the unique functional which satisfies $\hat \vol (x_{X_0}^{p-1} \nu^p) = 1$, in order to establish that our $\varepsilon$ is the correct one. We formulate this as the next lemma:

\begin{lemma}
    For the proposed
    $$\varepsilon=x_{X_0}^{p-1}\theta_0^{p-1}\cdots \theta_n^{p-1},$$
    we have $\hat \vol(\varepsilon \nu)=1$.
\end{lemma}

\begin{proof}
Fix some positively ordered facet $F=[v_0,\ldots,v_n]$ and let $\nu=\det(\theta_F)x_F \in \tk [X_0]_s$. It is enough to prove the congruence
$$\varepsilon\nu
\equiv
x_{X_0}^{p-1}\nu^p
\mod (I^{[p]}+\theta^{[p]}),$$
because then applying $\hat\vol$ gives
$$\hat\vol(\varepsilon\nu)
=
\hat\vol(x_{X_0}^{p-1}\nu^p)
=
1$$
by the definition of $\hat\vol$. We have
$$\varepsilon\nu
=
x_{X_0}^{p-1}\det(\theta_F)x_F
\theta_0^{p-1}\cdots\theta_n^{p-1}.$$
For every $i$, write
$$\theta_i=\theta_i^F+\theta_i^{F^c},
\qquad
\theta_i^F=\sum_{v\in F}\theta_{iv}x_v,
\qquad
\theta_i^{F^c}=\sum_{v\in X_0 \setminus F}\theta_{iv}x_v.$$
In the product
$$x_{X_0}^{p-1}x_F\theta_0^{p-1}\cdots\theta_n^{p-1},$$
any term which contains a variable $x_v$ with $v\notin F$ coming from one of the $\theta_i^{F^c}$ is divisible by $x_{F\cup\{v\}}^p$. Since $F$ is a facet, $F\cup\{v\}$ is not a face of $X$, hence $x_{F\cup\{v\}}^p\in I^{[p]}$. Therefore, modulo $I^{[p]}$, we may replace every $\theta_i$ by $\theta_i^F$. Thus
$$x_{X_0}^{p-1}x_F\theta_0^{p-1}\cdots\theta_n^{p-1}
\equiv
x_{X_0}^{p-1}x_F(\theta_0^F)^{p-1}\cdots(\theta_n^F)^{p-1}
\mod I^{[p]}.$$
Now we use the standard determinant identity
$$(\theta_0^F)^{p-1}\cdots(\theta_n^F)^{p-1}
\equiv
\det(\theta_F)^{p-1}x_F^{p-1}
\mod ((\theta_0^F)^p,\ldots,(\theta_n^F)^p).$$
Multiplying by $x_{X_0}^{p-1}x_F$, the congruence above gives
$$x_{X_0}^{p-1}x_F
(\theta_0^F)^{p-1}\cdots(\theta_n^F)^{p-1}
\equiv
x_{X_0}^{p-1}\det(\theta_F)^{p-1}x_F^p
\mod (I^{[p]}+\theta^{[p]}).$$
Here we used that
$$(\theta_i^F)^p=\theta_i^p-(\theta_i^{F^c})^p,$$
and the terms containing $(\theta_i^{F^c})^p$, after multiplication by $x_{X_0}^{p-1}x_F$, lie in $I^{[p]}$ by the same nonface argument as above. Therefore
$$\varepsilon\nu
=
x_{X_0}^{p-1}\det(\theta_F)x_F
\theta_0^{p-1}\cdots\theta_n^{p-1}
\equiv
x_{X_0}^{p-1}\det(\theta_F)^p x_F^p
\mod (I^{[p]}+\theta^{[p]}).$$
But $\nu^p=\det(\theta_F)^p x_F^p$. Hence
$$\varepsilon\nu
\equiv
x_{X_0}^{p-1}\nu^p
\mod (I^{[p]}+\theta^{[p]}).$$
Applying $\hat\vol$ gives
$$\hat\vol(\varepsilon\nu)
=
\hat\vol(x_{X_0}^{p-1}\nu^p)
=
1.$$
\end{proof}

We introduce a few notations. $\mathbb N = \{0,1,2,...\}$. If $a : X_0 \to \mathbb N$, then $x_a = \prod_{v \in X_0} x_v^{a(v)}$. We also use $|a| = \sum_{v \in X_0} a(v)$. Now we apply Theorem \ref{th:main}. For any $a: X_0 \to \mathbb N$ with $|a| = n+1$, we have
$$\vol (x_{a}) = \sum_{u \in M_{n+1}} ((x_{X_0}^{p-1} u^p) \circ (x_{X_0}^{p-1}\theta_0^{p-1}\cdots \theta_n^{p-1} x_a)) \vol(u)^p,$$ 
or, after simplifying,
$$\vol (x_{a}) = \sum_{u \in M_{n+1}} (u^p \circ (\theta_0^{p-1}\cdots \theta_n^{p-1} x_a)) \vol(u)^p.$$ 
We can expand further by the multinomial theorem:
$$
\theta_0^{p-1}\cdots\theta_n^{p-1}
=
(-1)^{n+1}
\sum_{|b_0|=\cdots=|b_n|=p-1}
\left(
\prod_{i=0}^{n}
\prod_{v\in X_0}
\frac{\theta_{iv}^{b_i(v)}}{b_i(v)!}
\right)
\prod_{v\in X_0}
x_v^{b_0(v)+\cdots+b_n(v)}.
$$
Substituting this into the previous identity gives
$$
\vol(x_a)
=
(-1)^{n+1}
\sum_{|b_0|=\ldots=|b_n|=p-1}
\left(
\prod_{i=0}^{n}
\prod_{v\in X_0}
\frac{\theta_{iv}^{b_i(v)}}{b_i(v)!}
\right)
\vol\left(x_{(a+b)/p}\right)^p,
$$
where the sum goes over all $b_0,\ldots,b_n$, each of type $b_i : X_0 \to \mathbb N$, with $|b_i| = p-1$, and $b(v) = b_0(v) + \ldots + b_n(v)$. The expression $x_{(a + b)/p}$ is declared to be equal to zero if $(a(v) + b(v))$ is not divisible by $p$ for some vertex $v \in X_0$. The sign $(-1)^{n+1}$ comes from the multinomial theorem and the fact that $(p-1)! = -1$ in characteristic $p$. This coincides with the formula given in \cite[Appendix 1, Theorem A.1]{AdiPoly}. Analogous formulas could be written for IDP lattice spheres and IDP lattice polytopes using the same approach.

\subsection{Minimal non-complete intersection example} \label{subsec:minimalnonci}
This is an example of a graded Artinian Gorenstein $\k$-algebra which is not a complete intersection. It is taken from \cite[Example 3.10]{Burke}, along with some homological data. We call this example minimal in the following sense: by the Buchsbaum--Eisenbud structure theorem \cite{BuchsbaumEisenbud}, every height $3$ Gorenstein ideal has an odd number of minimal generators. Since the case of $3$ generators is precisely the complete intersection case, every non-complete-intersection Gorenstein ideal of height $3$ has at least $5$ minimal generators.

Let $S = \k[x,y,z]$ and let
$$
g_0=x^2,\qquad g_1=-yz,\qquad g_2=xy+z^2,\qquad g_3=-xz,\qquad g_4=y^2.
$$
Let $I=(g_0,g_1,g_2,g_3,g_4)$ and $R=S/I$. The set $g_0,\ldots,g_4$ is a minimal homogeneous generating set of $I$. The algebra $R$ is a Gorenstein Artinian $\k$-algebra of socle degree $s=2$ and projective dimension $m = 3$. Let
$$
F^{\bullet} :=
(0 \to F^{-3} \to F^{-2} \to F^{-1} \to F^0 \to 0)
=
(0 \to S(-5) \to S(-3)^5 \to S(-2)^5 \to S \to 0)
$$
be a minimal free resolution of $R$ as an $S$-module, and let $\vol : R_2 \to \k$ be defined by $\vol(xy) = 1$. We denote the basis of $F^{-1}$ by $a_0,\ldots,a_4$ and the basis of $F^{-2}$ by $b_0,\ldots,b_4$; the basis of $F^{-3}$ is $c_0$. By $K^{\bullet}$ we mean the Koszul complex of the regular sequence $x,y,z$ in $\k[x,y,z]$. We name the basis of $K^{-1}$ as $e_0,e_1,e_2$; the basis of $K^{-2}$ is $e_1\wedge e_2, -e_0\wedge e_2,e_0\wedge e_1;$ and the basis of $K^{-3}$ is $e_0 \wedge e_1 \wedge e_2$.

% https://q.uiver.app/#q=WzAsMTIsWzUsMCwiMCJdLFs0LDAsIlMiXSxbMywwLCJTKC0yKV41Il0sWzIsMCwiUygtMyleNSJdLFsxLDAsIlMoLTUpIl0sWzAsMCwiMCJdLFs0LDEsIlMiXSxbNSwxLCIwIl0sWzMsMSwiUygtMSleMyJdLFsyLDEsIlMoLTIpXjMiXSxbMSwxLCJTKC0zKSJdLFswLDEsIjAiXSxbMSwwXSxbMiwxLCJkX3tGXzF9Il0sWzMsMiwiZF97Rl8yfSJdLFs0LDMsImRfe0ZfM30iXSxbNSw0XSxbMSw2LCI9Il0sWzYsN10sWzgsNiwiZF97S18xfSJdLFs5LDgsImRfe0tfMn0iXSxbMTAsOSwiZF97S18zfSJdLFsxMSwxMF0sWzIsOCwiTl8xIl0sWzMsOSwiTl8yIl0sWzQsMTAsIk5fMyJdXQ==
$$
\begin{tikzcd}
	0 & {S(-5)} & {S(-3)^5} & {S(-2)^5} & S & 0 \\
	0 & {S(-3)} & {S(-2)^3} & {S(-1)^3} & S & 0
	\arrow[from=1-1, to=1-2]
	\arrow["{d_F^{-3}}", from=1-2, to=1-3]
	\arrow["{N^{-3}}", from=1-2, to=2-2]
	\arrow["{d_F^{-2}}", from=1-3, to=1-4]
	\arrow["{N^{-2}}", from=1-3, to=2-3]
	\arrow["{d_F^{-1}}", from=1-4, to=1-5]
	\arrow["{N^{-1}}", from=1-4, to=2-4]
	\arrow[from=1-5, to=1-6]
	\arrow["{=}", from=1-5, to=2-5]
	\arrow[from=2-1, to=2-2]
	\arrow["{d_K^{-3}}", from=2-2, to=2-3]
	\arrow["{d_K^{-2}}", from=2-3, to=2-4]
	\arrow["{d_K^{-1}}", from=2-4, to=2-5]
	\arrow[from=2-5, to=2-6]
\end{tikzcd}
$$
With respect to the bases fixed above, the maps $N^\bullet:F^\bullet\to K^\bullet$ are given by
\[
N^{-1}=
\begin{pmatrix}
x & 0 & 0 & 0 & 0\\
0 & 0 & x & 0 & y\\
0 & -y & z & -x & 0
\end{pmatrix},
\qquad
N^{-2}=
\begin{pmatrix}
y & 0 & 0 & 0 & 0\\
0 & 0 & 0 & 0 & x\\
0 & -x & 0 & 0 & 0
\end{pmatrix},
\qquad
N^{-3}=(xy).
\]
The differentials of the Koszul complex are
\[
d_K^{-1}=(x\ \ y\ \ z),
\qquad
d_K^{-2}=
\begin{pmatrix}
0 & z & -y\\
-z & 0 & x\\
y & -x & 0
\end{pmatrix},
\qquad
d_K^{-3}=
\begin{pmatrix}
x\\
y\\
z
\end{pmatrix}.
\]
In particular, with this basis of $K^{-2}$, we have $d_K^{-3}=(d_K^{-1})^T$. The differentials of $F^\bullet$ are
$$
d_F^{-1}=
\begin{pmatrix}
x^2 & -yz & xy+z^2 & -xz & y^2
\end{pmatrix},
$$
$$
d_F^{-2}=
\begin{pmatrix}
0 & y & 0 & 0 & z\\
-y & 0 & x & z & 0\\
0 & -x & 0 & y & 0\\
0 & -z & -y & 0 & x\\
-z & 0 & 0 & -x & 0
\end{pmatrix},
\qquad
d_F^{-3} = (d_F^{-1})^T.
$$
This shows that
$$
N^{-3}(c_0)=(xy)e_0\wedge e_1\wedge e_2,
$$
so $\nu=xy$. Now assume that $\k$ is a field of characteristic $p$. Let
$$
I^{[p]} = (x^{2p},(-yz)^p,(xy+z^2)^p,(-xz)^p,y^{2p}).
$$
We introduce $\hat K^{\bullet} = K^{\bullet} \otimes_{Frob}S$, which is the Koszul complex of the sequence $x^p,y^p,z^p$. Thus $\hat K^{-1}$ has basis $\hat e_0,\hat e_1,\hat e_2$, $\hat K^{-2}$ has basis $\hat e_1\wedge \hat e_2,-\hat e_0\wedge \hat e_2, \hat e_0\wedge \hat e_1$ and $\hat K^{-3}$ has basis $\hat e_0\wedge\hat e_1\wedge\hat e_2$. Analogously, we define $\hat F^{\bullet} = F^{\bullet} \otimes_{Frob} S$ with basis
$$
\hat c_0;\qquad
\hat b_0,\ldots,\hat b_4;\qquad
\hat a_0,\ldots,\hat a_4;\qquad
1.
$$
We define $D : \hat K^{\bullet} \to K^{\bullet}$ to be the commutative dg-algebra morphism with
$$
D(\hat e_0)=x^{p-1}e_0,\qquad
D(\hat e_1)=y^{p-1}e_1,\qquad
D(\hat e_2)=z^{p-1}e_2.
$$
In order to get the matrices for $d_{\hat F}^i,\hat N^i,d_{\hat K}^i$, we just raise every coefficient to the power $p$.

\[
\begin{tikzcd}
	0 & {S(-5p)} & {S(-3p)^5} & {S(-2p)^5} & S & 0 \\
	0 & {S(-3p)} & {S(-2p)^3} & {S(-p)^3} & S & 0
	\arrow[from=1-1, to=1-2]
	\arrow["{d_{\hat F}^{-3}}", from=1-2, to=1-3]
	\arrow["{\hat N^{-3}}", from=1-2, to=2-2]
	\arrow["{d_{\hat F}^{-2}}", from=1-3, to=1-4]
	\arrow["{\hat N^{-2}}", from=1-3, to=2-3]
	\arrow["{d_{\hat F}^{-1}}", from=1-4, to=1-5]
	\arrow["{\hat N^{-1}}", from=1-4, to=2-4]
	\arrow[from=1-5, to=1-6]
	\arrow["{=}", from=1-5, to=2-5]
	\arrow[from=2-1, to=2-2]
	\arrow["{d_{\hat K}^{-3}}", from=2-2, to=2-3]
	\arrow["{d_{\hat K}^{-2}}", from=2-3, to=2-4]
	\arrow["{d_{\hat K}^{-1}}", from=2-4, to=2-5]
	\arrow[from=2-5, to=2-6]
\end{tikzcd}
\]
The differentials of $\hat K^\bullet$ are
\[
d_{\hat K}^{-1}=(x^p\ \ y^p\ \ z^p),
\qquad
d_{\hat K}^{-2}=
\begin{pmatrix}
0 & z^p & -y^p\\
-z^p & 0 & x^p\\
y^p & -x^p & 0
\end{pmatrix},
\qquad
d_{\hat K}^{-3}=
\begin{pmatrix}
x^p\\
y^p\\
z^p
\end{pmatrix}.
\]
The differentials of $\hat F^\bullet$ are
\[
d_{\hat F}^{-1}=
\begin{pmatrix}
x^{2p} & -y^p z^p & x^p y^p + z^{2p} & -x^p z^p & y^{2p}
\end{pmatrix},
\]
\[
d_{\hat F}^{-2}=
\begin{pmatrix}
0 & y^p & 0 & 0 & z^p\\
-y^p & 0 & x^p & z^p & 0\\
0 & -x^p & 0 & y^p & 0\\
0 & -z^p & -y^p & 0 & x^p\\
-z^p & 0 & 0 & -x^p & 0
\end{pmatrix},
\qquad
d_{\hat F}^{-3}=(d_{\hat F}^{-1})^T.
\]
The maps $\hat N^\bullet:\hat F^\bullet\to \hat K^\bullet$ are
\[
\hat N^{-1}=
\begin{pmatrix}
x^p & 0 & 0 & 0 & 0\\
0 & 0 & x^p & 0 & y^p\\
0 & -y^p & z^p & -x^p & 0
\end{pmatrix},
\qquad
\hat N^{-2}=
\begin{pmatrix}
y^p & 0 & 0 & 0 & 0\\
0 & 0 & 0 & 0 & x^p\\
0 & -x^p & 0 & 0 & 0
\end{pmatrix},
\qquad
\hat N^{-3}=(x^p y^p).
\]
We can define $\hat \vol : \hat R_{\hat s} \to \k$ as the unique functional such that
$$
\hat \vol(x^{p-1} y^{p-1} z^{p-1}\nu^p)
=
\hat \vol(x^{2p-1} y^{2p-1} z^{p-1})
=
1,
$$
where $\hat s = 2p + 3(p-1) = 5p - 3$. The morphism $P^{\bullet} : \hat F^{\bullet} \to F^{\bullet}$ can be defined as
\[
P^{-1}=
\begin{pmatrix}
x^{2p-2} & 0 & 0 & 0 & 0\\
0 & (-yz)^{p-1} & 0 & 0 & 0\\
0 & 0 & (xy+z^2)^{p-1} & 0 & 0\\
0 & 0 & 0 & (-xz)^{p-1} & 0\\
0 & 0 & 0 & 0 & y^{2p-2}
\end{pmatrix}.
\]

\[
P^{-2}=
\begin{pmatrix}
y^{2p-2}z^{p-1} & 0 & 0 & y^{p-2}\bigl(z(xy+z^2)^{p-1}-z^{2p-1}\bigr) & 0\\
0 & x^{p-1}(xy+z^2)^{p-1} & 0 & 0 & 0\\
0 & 0 & x^{p-1}y^{p-1}z^{p-1} & 0 & 0\\
0 & 0 & 0 & y^{p-1}(xy+z^2)^{p-1} & 0\\
0 & x^{p-2}\bigl(z(xy+z^2)^{p-1}-z^{2p-1}\bigr) & 0 & 0 & x^{2p-2}z^{p-1}
\end{pmatrix}.
\]

\[
P^{-3}=(\varepsilon),
\qquad
\varepsilon=x^{p-1}y^{p-1}z^{p-1}(xy+z^2)^{p-1}.
\]
A direct computation verifies that these matrices define a chain map
$$
P^\bullet:\hat F^\bullet\to F^\bullet
$$
lifting the quotient map $\hat R\to R$. Therefore, the Parseval--Rayleigh identity is
$$
\vol(w)
=
\sum_{u \in M_2}
\left(
(x^{p-1} y^{p-1} z^{p-1} u^p)
\circ
(x^{p-1}y^{p-1}z^{p-1}(xy+z^2)^{p-1}w)
\right)
\vol(u)^p,
$$
or, after simplifying,
$$
\vol(w)
=
\sum_{u \in M_2}
\left(
u^p \circ ((xy+z^2)^{p-1}w)
\right)
\vol(u)^p.
$$

\bibliographystyle{alpha}
\bibliography{Biblio}{}

\end{document}